\documentclass[11pt,a4paper,twoside,reqno]{amsart}      
\usepackage[T1]{fontenc} 
\usepackage[utf8]{inputenc}  
\usepackage{amsmath,amssymb,amsfonts,amsthm,amscd} 
\usepackage{bigints}
\usepackage{graphicx}                 
\usepackage{color}                    
\usepackage{ mathrsfs }
\usepackage{epstopdf}
\usepackage[margin=2.8cm]{geometry}
\usepackage{enumerate}
\usepackage[normalem]{ulem}
\usepackage{url}         
\usepackage{colonequals} 
\usepackage[foot]{amsaddr}

\numberwithin{equation}{section}

\newtheorem{theorem}{Theorem}[section]

\newtheorem{corollary}[theorem]{Corollary}
\newtheorem{lemma}[theorem]{Lemma}
\theoremstyle{definition}
\newtheorem{remark}[theorem]{Remark}

\newcommand{\smfrac}[2]{{\textstyle \frac{#1}{#2}}}

\def\!{\mathop{\mathrm{!}}}

\def\R{\mathbb{ R}}

\def\N{\mathbb{ N}}

\def\E{\mathcal{E}}

\def\G{\mathcal{G}}

\def\rcut{{r_{\rm cut}}}

\def\L{\Lambda}
\def\Rc{\mathcal{R}}

\def\Z{\mathbb{Z}}

\def\Wc{\dot{\mathcal{W}}^{\rm c}}
\def\Wi{\dot{\mathcal{W}}^{1,2}}
\def\WH{\dot{W}^{1,2}}

\def\mA{{\sf A}}
\def\mB{{\sf B}}

\def\bsep{\,\big|\,}
\def\b{\big}

\def\<{\langle}
\def\>{\rangle}

\def\E{\mathcal{E}}
\def\Ehom{\mathcal{E}^{\rm hom}}

\def\Rdef{R_{\rm def}}
\def\RS{R_{\rm S}}

\def\HH{\mathcal{W}}

\def\XXint#1#2#3{{\setbox0=\hbox{$#1{#2#3}{\int}$ }
\vcenter{\hbox{$#2#3$ }}\kern-.6\wd0}}

\def\Wper{\mathcal{W}^{\rm per}}   
\def\WHper{W^{\rm per}}   

\def\us{\bar{u}}                    

\def\Lhom{\Lambda^{\rm hom}}
\def\Dh{D^{\rm h}}

\newlength{\boxwidth}
\date \today

\title[Uniform Convergence of the Supercell Approximation]{Sharp Uniform Convergence Rate of the Supercell Approximation of a Crystalline Defect}

\author{Julian Braun}
\author{Christoph Ortner}
\thanks{JB and CO are supported by ERC Starting Grant 335120 and by EPSRC Grant EP/R043612/1}

\address[JB, CO]{Mathematics Institute, University of Warwick, Coventry CV4 7AL, UK.}

\subjclass[2010]{Primary: 65N12; Secondary: 65N15, 70C20, 74E15}
\keywords{Crystalline defect, supercell approximation, uniform convergence}

\begin{document}
\begin{abstract}
  We consider the geometry relaxation of an isolated point defect embedded in a
  homogeneous crystalline solid, within an atomistic description. We prove a
  sharp convergence rate for a periodic supercell approximation with respect to
  uniform convergence of the discrete strains.
\end{abstract}

\maketitle

\section{Introduction}
\label{sec:intro}
The high computational cost of atomistic material models requires that the
numerical geometry equilibration of crystalline defects is performed in small
computational cells, employing ``artificial boundary conditions'' to emulate the
crystalline far-field behaviour. Aside from the model error (due to
approximations in the potential energy surface) the main simulation error is
therefore the error induced by the boundary condition. In \cite{EOS2016} a
framework was introduced to rigorously estimate these errors for a variety of
defects and boundary conditions, including clamped and periodic, as well as to
estimate approximation errors in atomistic/continuum and QM/MM multi-scale
schemes \cite{2014-bqce,Olson2018-zz,2015-qmtb2}. All of these works control
the error in the canonical energy-norm.

In the present work we will prove the first sharp approximation error estimate
for crystal defect equilibration in the {\em maximum norm} for the strains in
dimension greater than one (see \cite{2008-M2AN-QC1D, 2010-ARMA-QCF} for
examples of results in one dimension and \cite{Lu2013-xj} for a result in three
but in the absence of defects). To highlight the main ideas required for this
extension in a transparent setting, we have chosen to restrict this work to
point defects embedded in an infinite homogeneous host crystal, under an
interatomic potential interaction. This system is approximated using a supercell
method with periodic boundary conditions, the most widely used scheme for
simulating point defects.

Our main motivation for this work is \cite{2018-entropy}, where we require a
sharp uniform convergence rate to obtain sharp convergence error estimates on
the vibrational entropy of a point defect, as well as
\cite{BHOdefectdevelopment} where our new results significantly simplify the
development of a multi-pole expansion theory for crystalline defects. However,
our results are also of independent interest, namely in any scenario where the
defect core geometry is of importance but not the far-field, in which case the
energy-norm severely overestimates the simulation error. Concretely, the
best-approximation error in the maximum norm is significantly smaller than in
the energy norm, and moreover, there is ample numerical evidence that the
best-approximation is indeed attained.

Unsurprisingly, and similarly as for maximum-norm error estimates for numerical
approximation of PDEs \cite{Rannacher1982-so,Dolzmann1999-jh}, our analysis
relies on ideas from elliptic regularity theory, specifically sharp
Green's function error estimates and a discrete Caccioppoli inequality.

Notably, our analysis applies not only to energy minimisers but to general
equilibria, in particular saddle points, which are important objects in
studying the mobility of crystalline defects. For these general equilibria,
our energy-norm error estimates are new as well.

{\bf Outline: } In \S~\ref{sec:results} we formulate the geometry equilibration
problem, introduce the supercell approximation, state our main convergence
results, and present numerical examples demonstrating that they are indeed
sharp. In \S~\ref{sec:proofs} we present the proofs.

\section{Results}
\label{sec:results}

\subsection{Geometry equilibration of a point defect}
\label{sec:results:model}
The reference configuration of a point defect embedded in a $d$-dimensional
homogeneous host crystal is given by a set $\L \subset \R^d$,
satisfying
\begin{itemize}
\item[{\bf (L)}] There exists $\Rdef > 0$, $\mA \in \R^{d \times d}$ invertible,
    such that \\ $\L \cap B_{\Rdef}$ is finite and
     $\L \setminus B_{\Rdef} = \mA \Z^d \setminus B_{\Rdef}$
\end{itemize}
We assume throughout that $d \geq 2$.

A lattice displacement is a function $u : \L \to \R^m$, where $m \geq 1$ is the
range dimension. Given an interaction cutoff radius $\rcut > 0$, the {\em
interaction range} at site $\ell$ is given by
\[
  \Rc_\ell := \{ n - \ell \,|\, n \in \L \} \cap B_\rcut.
\]
In particular, for $\ell > \rcut+\Rdef$ this is independent of $\ell$ and we write $\Rc_\ell=\Rc$.
The associated finite difference gradient is given by
\begin{equation*}
   Du(\ell) := \b( u(\ell+\rho) - u(\ell) \b)_{\rho\in\Rc_\ell}.
\end{equation*}
We assume $\rcut$ is large enough such that ${\rm span}\, \Rc_\ell = \R^d$ for all
$\ell \in \L$ and the graph with vertices $\L$ and edges $\{(\ell, \ell+\rho):
\ell \in \L, \rho \in \Rc_\ell\}$ is connected.

Of particular interest are compact and finite-energy displacements described,
respectively, by the spaces
\begin{equation}
   \label{eq: spaces}
   \begin{split}
   \Wc := \Wc(\L) &:= \b\{u:\L\to\R^m \bsep {\rm supp}(Du)~\text{is compact}\b\} \qquad \text{and}
   \\ \Wi :=  \Wi(\L) &:= \b\{u:\L\to\R^m \bsep \| Du \|_{\ell^2} < \infty \b\},
   \end{split}
\end{equation}
where
\begin{equation*}
   |Du(\ell)|^2 := \sum_{\rho \in \Rc_\ell} |D_\rho u(\ell)|^2
   \qquad \text{and} \qquad
   \|Du\|_{\ell^2(\L)} := \b\|\,|Du|\b\|_{\ell^2(\L)}.
\end{equation*}
The latter defines a semi-norm on both $\Wc$ and $\Wi$.

The homogeneous background lattice is $\Lhom := \mA \Z^d$, which of course
satisfies all foregoing conditions. Since we will frequently convert between a
defective lattice $\L$ and the associated homogeneous lattice $\Lhom$ we denote
the associated finite-difference operator by $\Dh u(\ell) = (D_\rho
u(\ell))_{\rho \in \Rc}$. We will normally identify $\Wc = \Wc(\L)$, $\Wi =
\Wi(\L)$ but make the domains explicit in the  case of the homogeneous system,
$\Wc(\Lhom), \Wi(\Lhom)$.

For each $\ell \in \L$ let $V_\ell \in C^4( (\R^m)^{\Rc_\ell})$,  with
$V_\ell(0) = 0$,  be the site-energy associated with the lattice site $\ell$,
then the total potential energy difference is given by
\begin{equation} \label{eq:defn E}
  \E(u) := \sum_{\ell \in \L} V_\ell(Du(\ell)).
\end{equation}
The re-normalisation $V_\ell(0) = 0$ is made for the sake of simplicity of
notation and signals that $\E$ is in fact an energy-difference. We assume that the interaction is homogeneous away from the defect, i.e., $V_\ell=V$ for all $\lvert \ell \rvert >  \rcut+\Rdef$, and that $V$ satisfies the natural point symmetry $V(A) = V((-A_{-\rho})_{\rho \in\Rc})$ for all $A \in (\R^m)^{\Rc}$.

$\E(u)$ is {\it a priori} only defined for $u \in \Wc$ or, slightly more
generally, for $u : \L \to \R^m$ with $|Du| \in \ell^1(\L)$. To define it on
$\Wi$, it is proven in \cite[Lemma 2.1]{EOS2016} that $\E : \Wc \to \R$ is
continuous with respect to the $\| D\cdot\|_{\ell^2}$-semi-norm and that there
exists a unique continuous extension to $\Wi$. We still call this extension
$\E$ and remark that, according to \cite[Lemma 2.1]{EOS2016}, $\E \in C^3(\Wi)$. This is only to justify our notation as we will never in fact reference the energy itself in this paper, but work directly with its first variation,
\[
  \< \delta \E(u), v \> = \sum_{\ell \in \L} \nabla V_\ell(Du(\ell)) \cdot Dv(\ell)
  \qquad \text{for } v \in \Wc.
\]

We are interested in equilibrium configurations, $\delta \E(\us) = 0$,
or written as a variational formulation,
\begin{equation} \label{eq:equil}
    \big\< \delta \E(\us), v \big\> = 0 \qquad \forall v \in \Wi  .
\end{equation}

We say that $u \in \Wi$ is {\em inf-sup stable} if $\delta^2 \E(u) : \Wi \to
(\Wi)'$ is an isomorphism which can, for example, be quantified via
\begin{equation} \label{eq:infsup}
  \inf_{ \substack{v \in \Wi \\ \|Dv\|_{\ell^2} = 1} }
  \sup_{ \substack{w \in \Wi \\ \|Dw\|_{\ell^2} = 1} }
  \big\< \delta^2 \E(u) v, w \big\> > 0.
\end{equation}
Of particular interest is the stability of solutions $u = \us$.

In addition, our analysis requires stability of the
homogeneous background crystal, a standard assumption in solid state physics
known as {\em phonon stability} \cite{Wallace1998-hx}, which in our notation
can be written as
\begin{equation} \label{eq:stabhom}
    \sum_{\ell \in \Lhom} \nabla^2 V(0)\big[ \Dh v(\ell), \Dh v(\ell) \big]
    \geq c_0 \| \Dh v \|_{\ell^2(\Lhom)}^2
    \qquad \forall v \in \Wi(\Lhom),
\end{equation}
for some $c_0 > 0$.
We assume throughout that \eqref{eq:stabhom} holds.

Under the lattice stability assumption \eqref{eq:stabhom}, it is shown in
\cite[Thm.\,1]{EOS2016} that any solution $\us \in \Wi$ to \eqref{eq:equil}
satisfies
\begin{equation} \label{eq:decay_ubar}
  |D^j \us(\ell)| \lesssim |\ell|^{1-d-j} \qquad \text{for } j = 1, 2, 3;
        \quad
        |\ell| \text{ sufficiently large.}
\end{equation}

\subsection{Supercell approximation}
\label{sec:results:supercell}
We consider a finite-domain approximation to \eqref{eq:equil} with periodic
boundary conditions, which we will call the {\em supercell approximation}. To
that end, let $\mB = (b_1, \dots, b_d) \in \R^{d \times d}$ invertible such that
$b_i \in \mA \Z^d$. For each $N \in \N$, let
\[
  \Lambda_N := \L \cap \mB (-N, N]^d
  \qquad \text{and} \qquad
  \L_N^{\rm per} := \bigcup_{\alpha \in 2N \Z^d} \big( \mB \alpha + \Lambda_N \big).
\]
Then the space of periodic displacements is given by
\[
  \Wper_N := \Wper_N(\L_N) := \big\{ u : \L_N^{\rm per} \to \R^m \,\big|\,
                u(\ell +  \mB \alpha) = u(\ell)
                \text{ for } \alpha \in 2 N \Z^d \big\}.
\]
For $u \in \Wper_N$ and for $N$ sufficiently large, the periodic potential
energy approximation is given by
\[
  \E_N(u) := \sum_{\ell \in \L_N} V_\ell(Du(\ell))
\]
and the resulting periodic supercell approximation to \eqref{eq:equil} by
\begin{equation} \label{eq:results:supercellapprox}
  \< \delta \E_N( \us_N ), v \> = 0 \qquad \forall v \in \Wper_N.
\end{equation}

\subsection{Sharp uniform convergence rate}
\label{sec:results:convergence}
While $\Wper_N \not\subset \Wi$, we can still compare $D\us$ and $D\us_N$ pointwise.

\begin{theorem} \label{th:main theorem}
Let $\us \in \Wi$ be an inf-sup stable solution to \eqref{eq:equil}, then
  there exist $C > 0$ such that, for $N$ sufficiently large, there are $\bar{u}_N \in \Wper_N$ satisfying \eqref{eq:results:supercellapprox} as well as
  \begin{align*}
    \| D\bar{u}_N - D \bar{u} \|_{\ell^\infty(\Lambda_N)}
    \leq C N^{-d}.
  \end{align*}
\end{theorem}

\begin{remark} \label{rem:general p}
  Since $\# \L_N \approx N^d$, applying Hölder's inequality to
  Theorem~\ref{th:main theorem} we obtain, for $p' = p / (p-1)$,
  $
    \| D\bar{u}_N - D\bar{u} \|_{\ell^p(\Lambda_N)}
    \leq C
    N^{-d/p'}.
  $
\end{remark}

\subsection{Numerical Tests}
\label{sec:numerical tests}
We implemented two numerical tests to confirm our analysis:
\begin{enumerate}
\item A vacancy in bulk W (bcc crystal structure), with interaction modelled
  by a Finnis-Sinclair (embedded atom) potential \cite{FSWang2013}.
\item A self-interstitial in bulk Cu (fcc crystal structure), with interaction
modelled by Morse pair-potential $\phi(r) = (e^{-2\alpha (r-1)} - 2 e^{-\alpha
(r-1)})\phi_{\rm cut}(r)$ with stiffness parameter $\alpha = 4$ and cubic spline
cut-off $\phi_{\rm cut}$ on the interval $[1.5, 2.3]$.
\end{enumerate}

In both cases, we choose a cubic computational cell: given the lattice
parameter $a_0$ (side-length of the unit cell in equilibrium) the matrix $\mB$
in \S~\ref{sec:results:supercell} is given by $\mB = a_0 I$.
The resulting equilibration problem \eqref{eq:results:supercellapprox} is then
solved using a preconditioned nonlinear conjugate gradient
algorithm~\cite{2016-precon1}. To estimate the error a numerical comparison
solution was computed with $N = \lceil 2.5 N_{\rm max} \rceil$, where
$N_{\rm max}$ denotes the largest $N$ chosen for the test.

\begin{figure}
  \centering
  \includegraphics[width=0.8\textwidth]{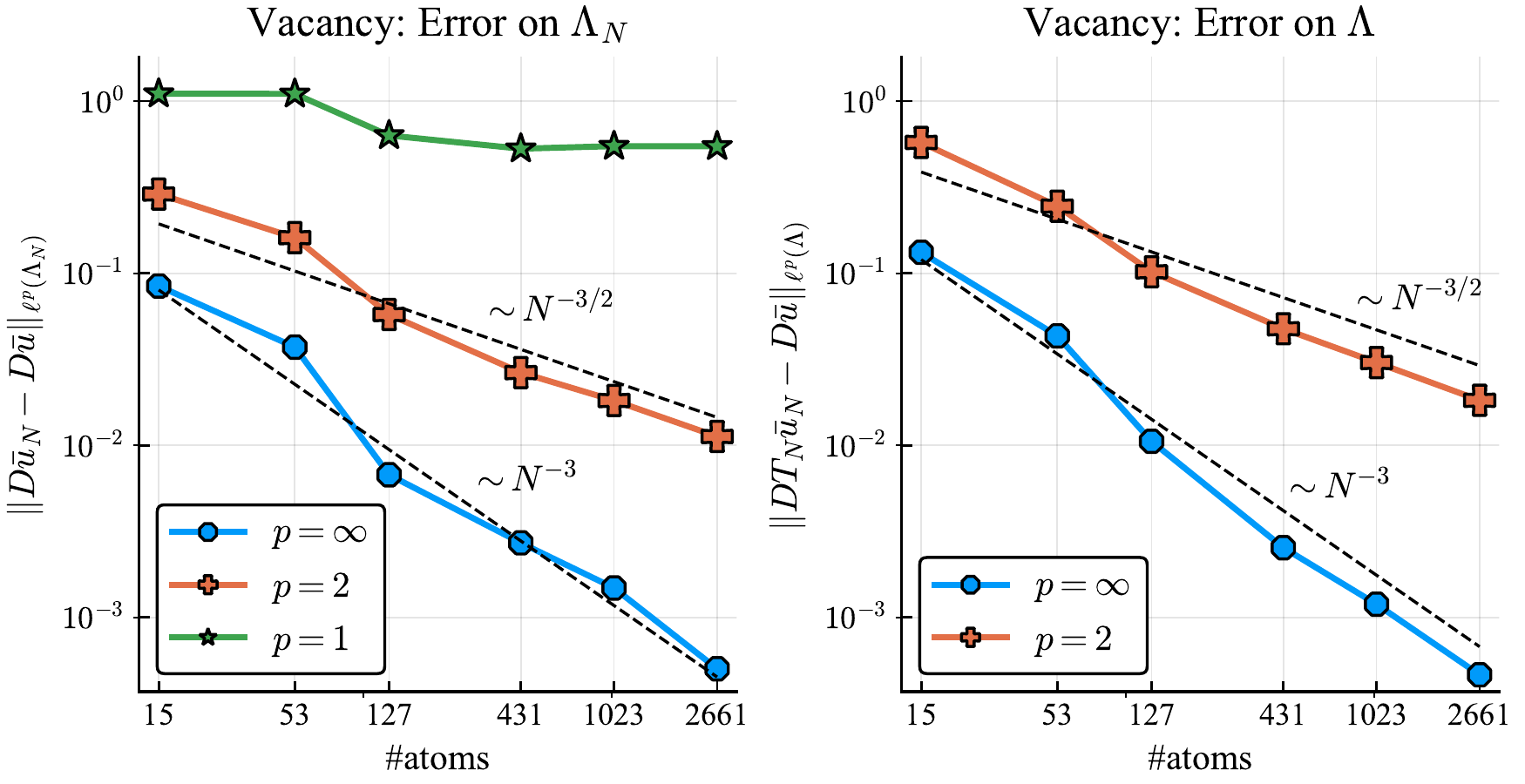}
  \caption{ \label{fig:errvacancy}
    Numerical confirmation of the convergence rates predicted by
    Theorem~\ref{th:main theorem}: vacancy in bulk W (bcc), under EAM
    interaction.
  }
\end{figure}

\begin{figure}
  \centering
  \includegraphics[width=0.8\textwidth]{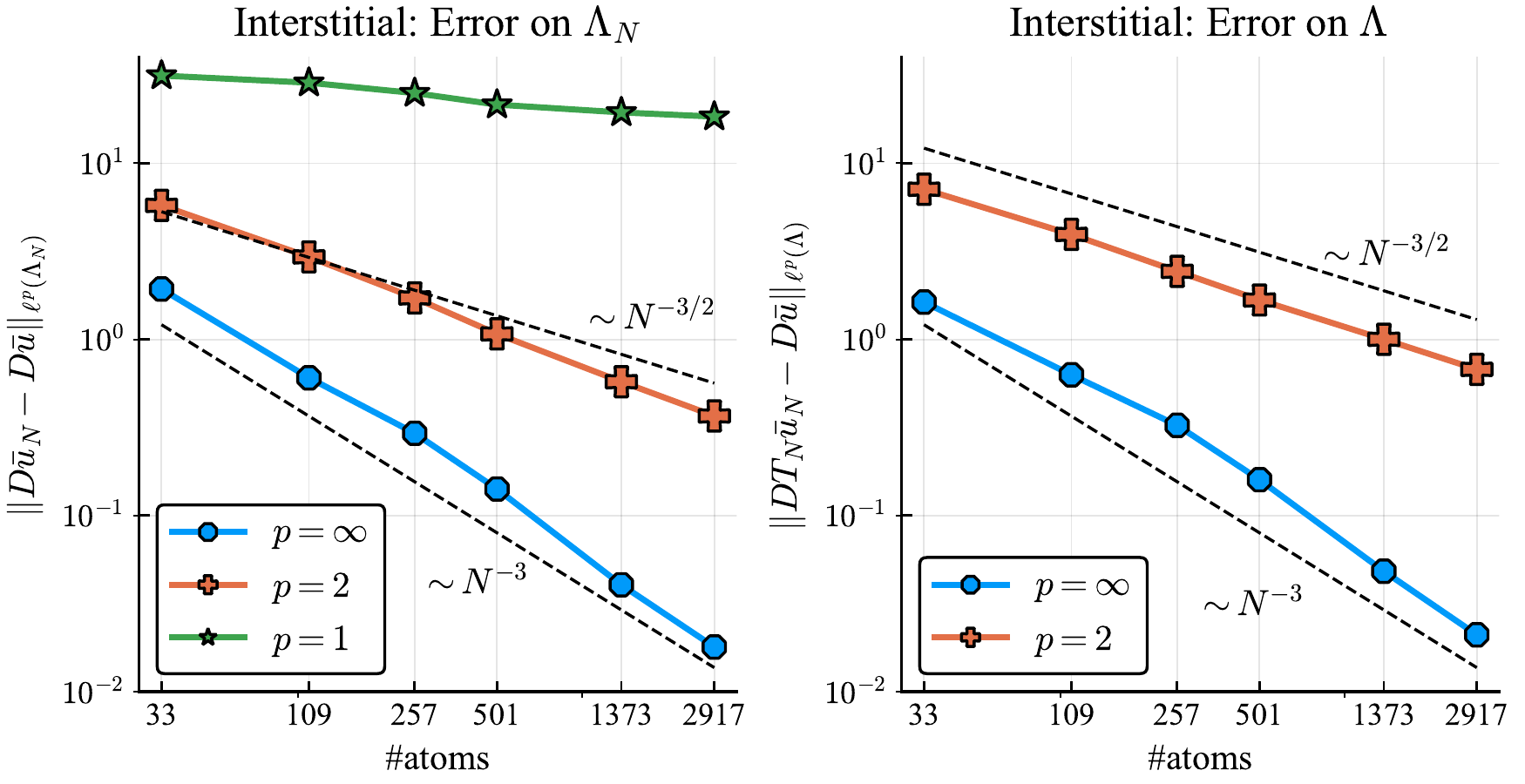}
  \caption{ \label{fig:errinter}
    Numerical confirmation of the convergence rates predicted by
    Theorem~\ref{th:main theorem}: interstitial in bulk Cu (fcc), under Morse
    interaction.
  }
\end{figure}

The results are shown in Figures~\ref{fig:errvacancy} and~\ref{fig:errinter}.
Although in both cases there is a mild pre-asymptotic behaviour visible, the
numerical errors follow closely the predicted rates. Note that we did not plot the errors on $\L$ with respect to the
$\|D\cdot\|_{\ell^1}$-seminorm since they are theoretically infinite but in
practise due to the finite domain of the comparison solution appear to converge
very slowly.

\subsection{Conclusion}
\label{sec:conclusion}
We haven given the first rigorous proofs of a sharp error estimate in the maximum norm
(for strains) for the relaxation of a crystalline defect under artificial
far-field boundary conditions.

Our restriction to point defects with periodic boundary conditions simplified
one key aspect of the analysis: the sharp error estimates for the Green's
function. Indeed, there are three fundamental ingredients in our analysis: (1)
an inf-sup condition which allowed us to treat general equilibria instead of
only minima; (2) a sharp error estimate for the Green's function; and (3) a
Caccioppoli estimate. Our arguments for (1) and (3) seem to be
 generic and can likely be generalised to other situations, in particular to
clamped boundary conditions for either point defects of dislocations.
Extending our error estimate for the Green's function is likely difficult in
general. However, whenever this can be achieved our results should be readily
extendable.

\section{Proofs}
\label{sec:proofs}
\def\vN{v_N}
\def\vNper{v_N^{\rm per}}
In \S\S~\ref{sec:proofs:aux}--\ref{sec:proofs:GN}, we establish auxiliary results,
mostly adapting existing ideas to our setting. The proof of inf-sup stability of
the periodic supercell approximation is given in \S~\ref{sec:infsup-proof}, and
the proof of the sharp uniform convergence estimate in \S~\ref{sec:uniform}.

\subsection{Auxiliary results}
\label{sec:proofs:aux}
An important technical tool that was used in \cite{EOS2016} for the error
analysis of the supercell approximation was a set of operators that enable us to
convert functions defined in $\L$ to functions defined on $\L_N$, and
vice-versa. The following results and their proofs are similar to those in
\cite{EOS2016}.

Let $Q_R :=  \mB (-R, R]^d$ and $\Lambda_R := \L \cap Q_R$ for any $R \in \N$. For general $R > 0$ we define $Q_R := Q_{\lceil R \rceil}$ and $\L_R
= \L_{\lceil R \rceil}$.

\begin{lemma}[Discrete Poincaré inequality] \label{lem:poincare}
There exist $r_{\rm P}, R_{\rm P}, C_{\rm P}>0$ such that for all $0<R_1<R_2$ with $R_1 \geq r_{\rm P}$, $R_2 - R_1 \geq r_{\rm P}$, $2 \leq p \leq \infty$,
 and $u : \L_{R_2+R_{\rm P}} \to \R^m$ we have
\begin{align} \label{eq:poincare}
  &\big\lVert u - \<u\>_{\L_{R_2} \setminus \L_{R_1}}
    \big\rVert_{\ell^p(\Lambda_{R_2} \setminus \Lambda_{R_1})}
  \leq R_2 C_{\rm P}
    \lVert Du \rVert_{\ell^p(\Lambda_{R_2 + R_{\rm P}} \setminus \Lambda_{R_1 - R_{\rm P}})} \\
  \label{eq:poincareaverage}
  & \text{with} \qquad \<u\>_{\L'} = \frac{1}{\lvert \Lambda' \rvert} \sum_{\ell \in \Lambda'} u(\ell).
\end{align}
\end{lemma}
\begin{proof}
  The restriction $R_1 \geq r_{\rm P}$ ensures that the defect region can be
  ignored. On can then apply \cite[Lemma 7.1]{EOS-preprint} and its proof
  verbatim to cubes instead of balls, which states that there exists $\tilde{a}$
  such that
  \[
    \|u - \tilde{a} \|_{\ell^2} \lesssim R_2 \lVert Du \rVert_{\Lambda_{R_2 + R_{\rm P}} \setminus \Lambda_{R_1 - R_{\rm P}}}.
  \]
    Since $\tilde{a} = \<u\>_{\L_{R_2}\setminus \L_{R_1}}$
  minimises the left-hand side, the stated result for $p = 2$ follows.

  For $p = \infty$, the result is elementary. For $2 < p < \infty$ it
  follows from the Riesz-Thorin interpolation theorem.
\end{proof}

Let $\eta_R \in C^2(\R^d; [0, 1])$ be a cut-off function satisfying
\begin{itemize}
\item $\eta_R(x) = 1$ for $x \in Q_{4R/6}$,
\item $\eta_R(x) = 0$ for $x \in \R^d \setminus Q_{5R/6}$,
\item $\lvert \nabla^j \eta_R \rvert \leq C R^{-j}$ for $j=1,2$.
\end{itemize}
Let $A_R := \L_{5R/6+\rcut} \setminus \L_{4R/6-\rcut}$ be a lattice annulus,
then, for $u: \L_R \to \R^m$ and $R \geq R_{\rm T} := \max\{2r_{\rm P}, 6 R_{\rm
P} + 6 \rcut\}$ we can define the truncation $T_R u \in \Wc$ by
\begin{equation*}
  T_R u(\ell)  :=
  \begin{cases}
    \eta_R \big(u- \<u\>_{A_R}\big), & \ell \in \L_R, \\
    0, & \text{otherwise.}
  \end{cases}
\end{equation*}
For $R \leq N$ we can extend $T_R u$ periodically with respect to $\Lambda_N$,
in which case we call it $T_{N,R}^{\rm per} u \in \Wper_N$. Moreover, we set
$T_N^{\rm per} := T_{N,N}^{\rm per}$. The following Lemma, while formulated in
terms of $T_R$ may also be applied to $T_{N,R}^{\rm per}$ and $T_{N}^{\rm per}$.

\begin{lemma} \label{th:TR_estimates}
There exists $C>0$ such that, for $R$ sufficiently large, $2 \leq p \leq \infty$,
 $u :\L_R \to \R^m$,
\begin{align}
    \label{eq:TR_estimates:global}
\lVert D T_R u \rVert_{\ell^p} &\leq C \lVert D u \rVert_{\ell^p(\Lambda_R)},\\
\label{eq:TR_estimates:err1}
\lVert D T_R u -Du \rVert_{\ell^p(\Lambda_R)}  &\leq C \lVert D u \rVert_{\ell^p(\Lambda_R \setminus \Lambda_{R/2})}, \qquad \text{and} \\
\label{eq:TR_estimates:err2}
\lVert D^2 T_R u -D^2 u \rVert_{\ell^p(\Lambda_R)}  &\leq C \lVert D^2 u \rVert_{\ell^p(\Lambda_R \setminus \Lambda_{R/2})}
  + C R^{-1} \| Du\|_{\ell^2(\L_R \setminus \L_{R/2})}.
\end{align}
\end{lemma}
\begin{proof}
Since
\[ D_\rho T_R u(\ell)  = \eta_R(\ell+\rho) D_\rho u(\ell)
    + D_\rho \eta_R(\ell) \big(u(\ell)-\<u\>_{A_R}\big),\]
we can use Lemma \ref{lem:poincare} and $R \geq R_{\rm T}$ to see that
\begin{align*}
\lVert D T_R u \rVert_{\ell^p} &\lesssim \lVert D u \rVert_{\ell^p(\Lambda_R)}
    + \frac{1}{R} \lVert u - \<u\>_{A_R} \rVert_{\ell^p(A_R)}
\lesssim \lVert D u \rVert_{\ell^p(\Lambda_R)},
\end{align*}
and
\begin{align*}
\lVert D T_R u -Du \rVert_{\ell^p(\Lambda_R)} &\lesssim \lVert D u \rVert_{\ell^p(\Lambda_R \setminus \Lambda_{R/2})} + \frac{1}{R} \lVert u - \<u\>_{A_R} \rVert_{\ell^p(A_R)}
\lesssim \lVert D u \rVert_{\ell^p(\Lambda_R \setminus \Lambda_{R/2})}.
\end{align*}
This establishes \eqref{eq:TR_estimates:global} and
\eqref{eq:TR_estimates:err1}. The proof of \eqref{eq:TR_estimates:err2} is
analogous.
\end{proof}

As an immediate corollary of Lemma~\ref{th:TR_estimates} we obtain pointwise
estimates on $T_R \us$.

\begin{corollary} \label{th:decay_TN_ubar}
  Let $\us \in \Wi$ be a solution to \eqref{eq:equil}, then there exists
  $C > 0$ such that for all $R > 0, \ell \in \L$ and $j = 1, 2$,
  \begin{equation} \label{eq:decay_TN_ubar}
      \big| D^j T_R \us(\ell) | \leq C (1+|\ell|)^{1-d-j}.
  \end{equation}
\end{corollary}
\begin{proof}
  The case $j = 1$ is an immediate consequence of \eqref{eq:decay_ubar} and
  \eqref{eq:TR_estimates:err1}, case $p = \infty$. The case $j = 2$
  follows from \eqref{eq:decay_ubar} and \eqref{eq:TR_estimates:err2}.
\end{proof}

\subsection{The Homogeneous Problem}
\label{sec:hom}
The proof of the sharp uniform convergence rates requires sharp estimates on the
Green's function for the homogeneous supercell. In preparation for these, we
first introduce some notation to effectively translate between the defective and
homogeneous problems.

Recall from \S~\ref{sec:results} that $\Lhom = \mA \Z^d$, and analogously let
$\Lhom_N = Q_N \cap \Lhom$, then we define the associated potential energies
by
\[
    \E^{\rm hom}(u):= \sum_{\ell \in \Lhom} V(\Dh u(\ell)),
    \qquad \text{and} \qquad
    \E^{\rm hom}_N(u):= \sum_{\ell \in \Lhom_N} V(\Dh u(\ell)),
\]
for, respectively, $u \in \Wc(\Lhom)$ and $u \in \HH^{\rm per}_N(\Lhom_N)$. Of
course, $\E^{\rm hom}, \E_N^{\rm hom}$ have the same regularity properties as
$\E$, listed in \S\,\ref{sec:results:model}.

Moreover, phonon stability \eqref{eq:stabhom} can now be written as
\[
  \big\< \delta^2 \E^{\rm hom}(0)v,v \big\> \geq c_0 \lVert \Dh v \rVert_{\ell^2(\Lhom)}^2 \qquad \forall v \in \Wi(\Lhom).
\]
As a consequence there exists a {\em lattice Green's function}.
\begin{lemma} \label{th:LGF}
  There exists a {\em lattice Green's function} $\G : \Lhom \to
  \R^{m \times m}$ satisfying
  \begin{align}
    \notag
    &\big\< \delta^2 \E^{\rm hom}(0) (\G e_i),v \big\> = v_i(0)
    \qquad \forall v \in \Wc(\Lhom), \qquad \text{and} \\
    \label{eq:prfs:decay_G}
    &\lvert (\Dh)^j \G(\ell) \rvert \leq C_j
    (1+ \lvert \ell \rvert)^{2-d-j}
    \qquad \forall j \geq 1, \ell \in \Lhom.
  \end{align}
  Furthermore, $\G(\ell) = \G(-\ell)$ for all $\ell \in \Lhom$.
\end{lemma}
\begin{proof}
   The two claims in \eqref{eq:prfs:decay_G} are proven in \cite[Lemma
   12]{EOS2016}. The point  symmetry of $\G$ is an immediate consequence of the
   fact that the Fourier symbol satisfies $\hat{\G}(k)=\hat{\G}(-k)$.
\end{proof}

To compare displacements $u$ of the homogeneous and the defect problem, we
define linear operators $S_N^{\rm hom} : \Wper_N(\L_N) \to \Wper_N(\Lhom_N)$
and $S_N^{\rm def} : \Wper_N(\Lhom_N) \to \Wper_N(\L_N)$ by fixing any
 $\ell_0 \in \L \setminus B_{\Rdef}$ and then letting
\begin{align*}
S_N^{\rm hom} u (\ell) &= \begin{cases}
        u(\ell), & \ell \in \Lhom \setminus B_{\Rdef}\\
        u(\ell_0), & \ell \in \Lhom \cap B_{\Rdef}
    \end{cases}
    \qquad \text{for } u \in \Wper_N(\L_N), \text{ and} \\
S_N^{\rm def} u (\ell) &= \begin{cases}
        u(\ell), & \ell \in \L \setminus  B_{\Rdef}\\
        u(\ell_0), & \ell \in \L \cap B_{\Rdef}
    \end{cases}  \qquad \text{for } u \in \Wper_N(\Lhom_N).
\end{align*}

In particular, we have the following lemma as an immediate consequence of these
definitions.

\begin{lemma} \label{th:properties_S}
  For some $\RS \geq \Rdef + \rcut$ sufficiently large, we have
    $\Dh S_N^{\rm hom} u =Du$ and $D S_N^{\rm def} u =\Dh u$ for $\lvert \ell \rvert > \Rdef + \rcut$ as well as the estimates
  \begin{align*}
    \lvert \Dh S_N^{\rm hom} u(\ell) \rvert &\leq C \lVert D u \rVert_{\ell^\infty(B_{\RS}\cap \L_N)} \qquad \forall \ell \in \Lhom \cap B_{\Rdef + \rcut}, \quad  \text{and} \\
    \lvert D S_N^{\rm def} u(\ell) \rvert &\leq C \lVert \Dh u \rVert_{\ell^\infty(B_{\RS}\cap \L_N^{\rm hom})} \qquad \forall \ell \in \L \cap B_{\Rdef + \rcut}.
  \end{align*}
\end{lemma}

\begin{remark}
  The operators $S_N^{\rm def}$ and $S_N^{\rm hom}$ are not ``optimized''
  for practical purposes, which likely leads to poor constants in some of our
  estimates. However, we only use them as a technical tool in the proofs,
  and are only concerned with rates. For specific defect structures, more
  natural operators $S_N^{\rm def}, S_N^{\rm hom}$ are easily constructed.
\end{remark}

The definition and all properties in Lemma \ref{th:properties_S} directly translate to analogous operators $S^{\rm hom} : \Wi(\L) \to \Wi(\Lhom)$ and $S^{\rm def} : \Wi(\Lhom) \to \Wi(\L)$ as well.

\subsection{Periodic Green's Function}
\label{sec:proofs:GN}
We begin by recalling that phonon stability \eqref{eq:stabhom} also ensures the stability of the homogeneous periodic problem:
\begin{lemma} \label{lem:periodicstable}
  \cite[Thm. 3.6]{2012-M2AN-CBstab} For all $N > 0$ we have
\[  \big\< \delta^2 \E^{\rm hom}_N(0) v,v \big\>
    \geq c_0 \lVert \Dh v \rVert_{\ell^2(\Lhom_N)}^2 \qquad \forall v \in \HH^{\rm per}_N(\Lhom_N).
\]
\end{lemma}

In particular, for every $f : \Lhom_N \to \R^m$ with $\sum_{\ell \in \Lhom_N}f(\ell) =0$,
there exists a unique $u \in \HH^{\rm per}_N(\Lhom_N)$ with
$\sum_{\ell \in \Lhom_N}u(\ell) =0$ such that
\[
  \big\< \delta^2 \E^{\rm hom}_N(0) u,v \big\> = (f, v)_{\ell^2}
  \qquad \forall v \in \HH^{\rm per}_N(\Lhom_N).
\]
The {\em periodic Green's function} $\G_N : \Lhom_N \to \R^{m \times m}$ is then
defined by the equation
\begin{align*}
\big\< \delta^2 \E^{\rm hom}_N(0) (\G_N e_i), v \big\>
  &= v_i(0) -
  \frac{1}{\lvert \Lhom_N \rvert}\sum_{\ell \in \Lhom_N}v_i(\ell)
  \qquad \forall v \in \HH^{\rm per}_N(\Lhom_N)\\
  &=: \big( \delta_0 e_i - \frac{1}{\lvert \Lhom_N \rvert} {\bf 1} e_i , v \big)_{\ell^2}
\end{align*}

To estimate the decay of $\G_N$ we relate $\G$ and $\G_N$.

\begin{lemma}   \label{th:error estimate FN}
For every $j \geq 1$ there exist constants $C_1, C_2$, independent of $N$, such that
   \begin{align*}
      \big\| (\Dh)^j \G - (\Dh)^j \G_N \big\|_{\ell^\infty(\Lhom_{N})} &\leq C_1 N^{2-d-j}, \qquad
            \text{and in particular} \\
      |(\Dh)^j \G_N(\ell)| &\leq C_2 (1+{\rm dist}(\ell, 2N\mB \Z^d))^{2-d-j} \qquad
      \forall\, \ell \in \Lhom.
   \end{align*}
\end{lemma}
\begin{proof}
First, we note that
\begin{equation}\label{eq:SD3Geq0}
\sum_{\ell \in \Lhom} (\Dh)^j \G(\ell) =0 \qquad \forall j \geq 3,
\end{equation}
which is straightforward to prove due to the gradient structure. (For $j < 3$,
$(\Dh)^j \G(\ell)$ does not decay fast enough to even define this sum.)


Fix $i \in \{1,\dots,m\},\rho_1,\rho_2,\rho_3 \in \Rc$ and let
\[
    w(\ell):= \sum_{z \in \Z^d} D_{\rho_1}D_{\rho_2}D_{\rho_3} \G(\ell + 2N \mB z) e_i
    \qquad \forall \ell \in \Lhom.
\]
Due to the decay \eqref{eq:prfs:decay_G}, the sum exists. Moreover, $w$ is
$\Lhom_N$-periodic, satisfies
\[
    \big\< \delta^2 \E^{\rm hom}_N(0) w,v \big\>
    = (D_{\rho_1}D_{\rho_2}D_{\rho_3} \delta_0 e_i, v)_{\ell^2}
\]
and according to \eqref{eq:SD3Geq0} also $\sum_{\ell \in \Lhom_N} w(\ell) = 0$.

Since $D_{\rho_1}D_{\rho_2}D_{\rho_3} \G_N e_i$ solves the same equation and
has average zero as well, we can therefore deduce that
\begin{equation*}
  (\Dh)^3 \G_N = \sum_{z \in \Z^d} (\Dh)^3 \G(\ell + 2N \mB z).
\end{equation*}
Consequently, for $j \geq 3$ and $\ell \in \Lhom_N$,
\begin{align*}
\big\lvert (\Dh)^j \G(\ell) - (\Dh)^j \G_N(\ell) \big\rvert &= \Big\lvert \sum_{z \in \Z^d \setminus \{0\}} (\Dh)^j \G(\ell + 2N \mB z) \Big\rvert\\
&\lesssim \sum_{z \in \Z^d \setminus \{0\}} (1+ \lvert \ell + 2N \mB z \rvert)^{2-d-j}\\
&\lesssim N^{2-d-j} \sum_{z \in \Z^d \setminus \{0\}} \lvert \mB^{-1}\ell/N + 2z \rvert^{2-d-j}\\
&\lesssim N^{2-d-j},
\end{align*}
where we used that the series converges due to $j \geq 3$ and the estimate is
uniform due to the uniform lower bound $\lvert \mB^{-1}\ell/N + 2z \rvert \geq
1$.

It remains to establish the estimate for $j=1,2$, which we will obtain
from a discrete Poincar\'e inequality: For all $g : \Lhom_{N} \to \R^m$ we clearly have
   \[\lvert g(x) - g(y) \rvert \leq C N \lVert \Dh g \rVert_{\ell^\infty(\Lhom_{N})}
   \qquad  \forall  x, y \in \Lhom_{N} \]
hence it immediately follows that
\begin{equation} \label{eq:proof:err est FN:poincare}
  \| g - \<g\>_{\Lhom_N} \|_{\ell^\infty(\Lhom_{N})} \leq C N \| \Dh g \|_{\ell^\infty(\Lhom_{N})},
\end{equation}
   where $\<g\>_{\Lhom_N} = \lvert \Lhom_N \rvert^{-1} \sum_{\ell \in \Lhom_N} g(\ell)$.

Fix $\rho, \sigma \in \Rc$ and let $C_N :=  \< D_\rho D_\sigma \G-D_\rho D_\sigma\G_N \>_{\Lhom_N}$, then combining the estimate for $j=3$ and \eqref{eq:proof:err est FN:poincare}
   we obtain
   \begin{align*}
      \| D_\rho D_\sigma \G- D_\rho D_\sigma \G_N \|_{\ell^\infty(\Lhom_{N})}
      &\leq   \| D_\rho D_\sigma \G - D_\rho D_\sigma \G_N - C_N \|_{\ell^\infty(\Lhom_{N})}
               + | C_N | \\
      & \lesssim N \| \Dh D_\rho D_\sigma \G - \Dh D_\rho D_\sigma \G_N \|_{\ell^\infty(\Lhom_{N})}
               + | C_N |\\
      & \lesssim N^{-d} + \big| C_N \big|.
   \end{align*}
   It thus remains to estimate $C_N$.

   Periodicity of $\G_N$ implies that $\< D_\rho D_\sigma \G_N \>_{\Lhom_N} = 0$, hence,
   \[
      C_N = \lvert \Lhom_N \rvert^{-1} \sum_{\ell \in \Lhom_N} D_\rho D_\sigma \G (\ell).
   \]
   Using discrete summation by parts we see that
   \begin{align*}
      |C_N|
      &= |\Lhom_N|^{-1}
      \bigg| \sum_{\ell \in (\Lhom_N + \rho) \setminus \Lhom_N} D_\sigma \G (\ell)
      - \sum_{\ell \in \Lhom_N \setminus (\Lhom_N + \rho)} D_\sigma \G (\ell) \bigg| \\
      &\lesssim N^{-d} N^{d-1} N^{1-d} = N^{-d}.
   \end{align*}
   This establishes the result for $j=2$.

   To prove the estimate for $j = 1$, we can repeating the same argument on just
   $D_\rho  \G- D_\rho  \G_N$, to obtain
   \begin{equation*}
    \| D_\rho \G- D_\rho \G_N \|_{\ell^\infty(\Lhom_{N})}
      \lesssim N^{1-d} + N^{-d} \bigg| \sum_{\ell \in \Lhom_N} D_\rho \G (\ell) \bigg|.
   \end{equation*}
   From here on, however, we need to argue differently and in particular exploit
   cancellations due to symmetries in the Green's function, to avoid logarithmic
   terms. To that end, we define the point symmetric extension
   \[ \Lhom_{s,N} = \Lhom \cap \mB[-N,N]^d\]
   of $\Lhom$ (i.e., $-\Lhom_{s,N} = \Lhom_{s,N}$)
   and use $\G(\ell)=\G(-\ell)$ to calculate
\begin{align*}
\bigg| \sum_{\ell \in \Lhom_N} D_\rho \G (\ell) \bigg| &\lesssim \bigg| \sum_{\ell \in \Lhom_{s,N}} D_\rho \G (\ell) \bigg| + N^{1-d} N^{d-1}\\
      &= \bigg| \sum_{\ell \in (\Lhom_{s,N} + \rho) \setminus \Lhom_{s,N}} \G (\ell)
      - \sum_{\ell \in \Lhom_{s,N} \setminus (\Lhom_{s,N} + \rho)} \G (\ell) \bigg|+1\\
      &=\bigg| \sum_{\ell \in (\Lhom_{s,N} + \rho) \setminus \Lhom_{s,N}} \G (-\ell)
      - \sum_{\ell \in \Lhom_{s,N} \setminus (\Lhom_{s,N} + \rho)} \G (\ell) \bigg| + 1 \\
      &= \bigg| \sum_{\ell \in (\Lhom_{s,N} - \rho) \setminus \Lhom_{s,N}} \G (\ell)
      - \sum_{\ell \in \Lhom_{s,N} \setminus (\Lhom_{s,N} + \rho)} \G (\ell) \bigg| + 1\\
      &= \bigg| \sum_{\ell \in (\Lhom_{s,N}) \setminus (\Lhom_{s,N}+\rho)} \G (\ell-\rho ) - \G(\ell) \bigg| + 1\\
      &\leq \sum_{\ell \in (\Lhom_{s,N}) \setminus (\Lhom_{s,N}+\rho)} \lvert D\G (\ell) \rvert + 1\\
      &\lesssim N^{d-1} N^{1-d} + 1\\
      &\lesssim 1.
   \end{align*}
   Hence, $\| D_\rho \G- D_\rho \G_N \|_{\ell^\infty(\Lhom_{N})} \lesssim  N^{1-d}$.
\end{proof}

\subsection{Inf-sup stability}
\label{sec:infsup-proof}
\def\WHH{W}
The first step in the error analysis of the supercell approximation is to
establish that it inherits inf-sup stability \eqref{eq:infsup}. This is a
generalisation of the result in \cite[Theorem 7.7]{EOS-preprint} that the
supercell approximation inherits positivity of the Hessian operator, a more
stringent notion of stability suitable only for minimisers.

In the stability analysis it is convenient to factor out constants from $\Wi$
and $\Wper_N$. Let $\HH_0 \subset \Wi$ and $\HH_{N,0} \subset \Wper_N$ denote
the $m$-dimensional subspaces of all constant functions, then we define
\[
    \WH := \Wi / \HH_0 \qquad \text{and} \qquad
    \WHper_N := \Wper_N / \HH_{N,0}.
\]
The associated equivalence classes of a function $u \in \Wi, \Wper_N$ are
denoted by $[u]$, however, whenever an expression is independent of constants,
we will abuse notation and identify $[u] \equiv u$, for example, $D[u] =
Du$. The inner products associated with $\WH, \WHper_N$ are then defined by
\[
    ( v, w )_{\WH} = ( Dv, Dw )_{\ell^2(\L)}  \qquad \text{and}
    \qquad
    ( v, w )_{\WHper_N} = ( Dv, Dw )_{\ell^2(\L_N)},
\]
and turn these factor spaces into Hilbert spaces.

For the proofs of the following results, recall that we made the standing
assumption \eqref{eq:stabhom} that the homogeneous reference lattice is stable.
Without this (standard) assumption the negative eigenspace identified in
Lemma~\ref{th:infsup-proof:prelim} need not be finite-dimensional, and this
would make our strategy infeasible.

\begin{lemma} \label{th:infsup-proof:prelim}
  (i) For all $u \in \Wi$, there exists a subspace $\HH_1 \subset \Wi$ with
  finite co-dimension such that
  \[
      \big\< \delta^2 \E(u) v, v \big\> \geq \smfrac12 c_0 \|Dv\|_{\ell^2}^2
      \qquad \forall v \in \HH_1.
  \]

  (ii) If, in addition, $u$ is inf-sup stable \eqref{eq:infsup} then there
  exists $c_1 > 0$ and an orthogonal decomposition $\WH = \WHH_- \oplus \WHH_+$
  with ${\rm dim}(\WHH_-) = q$ finite and
  \[
    \pm \big\< \delta^2 \E(u) v, v \big\> \geq c_1 \|Dv\|_{\ell^2}^2
    \qquad \forall v \in \WHH_\pm.
  \]
  Moreover, we may choose $\WHH_- = {\rm span}\{\psi_1, \dots, \psi_q \}$, where
  $\psi_j$ are eigenfunctions of $\delta^2 \E(u)$ in the $\WH$ sense.
\end{lemma}
\begin{proof}
  (i) Let $\HH_R := \{ v \in \Wi \,|\, Dv|_{B_R} = 0 \}$, then for $v \in \HH_R$,
  and for $R > \Rdef +\rcut$,
  \begin{align*}
    \big| \big\< \delta^2 \E(u) v, v \big\>
          - \big\< \delta^2 \Ehom(0) S^{\rm hom} v, S^{\rm hom} v \big\> \big|
    &\lesssim
    \sum_{\ell \in \L \setminus B_{R}}
      \big| \nabla^2 V(Du(\ell)) - \nabla^2 V(0) \big| \, | Dv(\ell)|^2 \\
    & \lesssim
      \| Du \|_{\ell^\infty(\L \setminus B_{R})} \,
       \|Dv\|_{\ell^2}^2
     \lesssim \epsilon_R \|Dv\|_{\ell^2}^2.
  \end{align*}
  where $\epsilon_R \to 0$ as $R \to \infty$ since $Du \in \ell^2$.
  Phonon stability \eqref{eq:stabhom} then implies that, for $R$ sufficiently large,
  \[
    \big\< \delta^2 \E(u) v, v \big\> \geq \smfrac12 c_0 \| Dv \|_{\ell^2}^2.
  \]
  Since the co-dimension of $\HH_R$ is finite, statement (i) follows
  with $\HH_1 = \HH_R$.

  (ii) Since $\delta^2 \E(u)$ is a symmetric, continuous bilinear map on $\WH$,
  there is a unique linear, self-adjoint, bounded operator $A(u) \in L(\WH)$
  with $\< \delta^2 \E(u) v, w \> = (A(u)v,w)_{\WH}$ for all $v,w \in \WH$.
  Inf-sup stability \eqref{eq:infsup} implies that $A(u)$ is an isomorphism.
  Thus, the spectrum of $A(u)$ is real, bounded, and bounded away from $0$. In
  light of (i), the spectral subspace of the negative part of the spectrum is
  finite dimensional. The negative part of the spectrum thus consists of only
  finitely many eigenvalues (with multiplicity) $\lambda_1 \leq \cdots \leq
  \lambda_q$ and associated orthonormal eigenfunctions $\psi_j$, $1 \leq j \leq q$. Define
  $\HH_- := {\rm span}\{\psi_1, \dots, \psi_q\}$ to be that spectral subspace,
  and $\HH_+$ the orthogonal complement of $\HH_-$, then (ii) follows.
\end{proof}

\begin{lemma}  \label{th:inf-sup-N}
  Suppose that $u \in \Wi$ is inf-sup stable \eqref{eq:infsup} then, for $N$
  sufficiently large,
  \begin{equation} \label{eq:inf-sup-N}
     \inf_{\substack{v \in \Wper_N \\ \|Dv\|_{\ell^2} = 1}}
     \sup_{\substack{w \in \Wper_N \\ \|Dw\|_{\ell^2} = 1}}
     \b\< \delta^2 \E_N(T_{N}^{\rm per} u) v, w \b\> \geq \min(c_0/8, c_1/4).
  \end{equation}
\end{lemma}
\begin{proof}
  Let $u_N=T_{N}^{\rm per} u$, $H_N := \delta^2 \E_N(u_N)$ and $H := \delta^2 \E(u)$. We will
  consider the orthogonal decomposition $\WHper_N = \WHH_{N,+} \oplus \WHH_{N,-}$,
  where
  \[
      \WHH_{N,-} = {\rm span}\big\{ T_{N,N/2}^{\rm per} \psi_j \,|\, j = 1, \dots, q \big\}
  \]
  with $\psi_j$ the negative eigenfunctions of $H$ (cf.
  Lemma~\ref{th:infsup-proof:prelim}(ii)) and $\WHH_{N,+}$ its orthogonal
  complement. We will prove that $H_N$ is uniformly positive on $\WHH_{N,+}$ and uniformly negative on $\WHH_{N,-}$, which implies the stated inf-sup condition \eqref{eq:inf-sup-N}.

  If $v_N \in \WHH_{N,-}$ then $v_N = T_{N,N/2}^{\rm per} v$ for some $v \in
  \WHH_-$. In particular, $Dv_N(\ell) = 0$ for $\ell \in \L_N \setminus \L_{N/2}$
  and since also $Du_N(\ell) = Du(\ell)$ for all $\ell \in \L_{N/2}$ we obtain
  \begin{align*}
      \big\< H_N  T_{N,N/2}^{\rm per} v, T_{N,N/2}^{\rm per} v \big\>
      &=
      \sum_{\ell \in \L_{N/2}}
      \nabla^2 V_\ell(D u_N(\ell))\big[ DT_{N,N/2}^{\rm per} v, DT_{N,N/2}^{\rm per} v \big] \\
      &=
      \sum_{\ell \in \L_{N/2}}
      \nabla^2 V_\ell(D u(\ell))\big[ DT_{N/2} v, DT_{N/2} v \big] \\
      &=
      \big\< H T_{N/2} v, T_{N/2} v \big\> \\
      &=
      \big\< H v, v \big\> + \big\< H \big(T_{N/2} v + v\big), T_{N/2}v - v \big\> \\
      &\leq
      - c_1 \|v\|_{\WH}^2 + C \| v \|_{\WH} \| T_{N/2} v - v \|_{\WH}.
  \end{align*}
  Since $\| T_{N/2}\psi_j - \psi_j \|_{\WH} \to 0$ for all $1 \leq j \leq q$, for a given $\epsilon$ we
  obtain $\| T_{N/2} v - v \|_{\WH} \leq \epsilon \|  v \|_{\WH}$ for all $N$ large enough uniformly in $v \in \WHH_-$. For $\epsilon$ small enough,
  \begin{equation}  \label{eq:prf-infsup-Hminus-bound}
      \< H_N v_N, v_N \> \leq
      (-c_1 + C\epsilon) \|v\|_{\WH}^2
      \leq \frac{-c_1 + C\epsilon}{(1 + \epsilon)^2} \|v_N\|_{\WHper}^2 \leq -c_1/2 \|v_N\|_{\WHper}^2.
  \end{equation}

  Next we prove uniform positivity of $H_N$ on $\WHH_{N,+}$, the complement of
  $\WHH_{N,-}$. This is a straightforward variation of the argument when $\WHH_{N,-} =
  \{0\}$ treated in \cite[Theorem 7.7]{EOS-preprint}. First, we take an increasing
  sequence $N_k$ such that
  \[ \lim_{k \to \infty}\min_{\substack{v \in \WHH_{N_k,+} \\ \lVert v \rVert_{\WHper_{N_k}} = 1}}
        \< H_{N_k} v, v \> = \liminf_{N \to \infty}\min_{\substack{v \in \WHH_{N,+} \\ \lVert v \rVert_{\WHper_N} = 1}}
        \< H_N v, v \>, \]
  and then choose
  \[
      v_k \in \arg\min_{\substack{v \in \WHH_{N_k,+} \\ \lVert v \rVert_{\WHper_{N_k}} = 1}}
        \< H_{N_k} v, v \>.
  \]

  Next, we want to choose a second sequence $R_k \uparrow \infty$, $R_k \leq N_k/4$, and
  decompose
  \[
      v_k = a_k + b_k := T_{N_k,R_k}^{\rm per} v_k + (I-T_{N_k,R_k}^{\rm per}) v_k.
  \]
  According to \cite[Lemma 7.8]{EOS-preprint} and \cite[Proof of Theorem 7.7]{EOS-preprint} , one can find a subsequence of $(N_k)_{k\in\N}$ (not relabelled) and $R_k \uparrow \infty$ sufficiently slowly such that
  \begin{equation} \label{eq:proof-infsup:vanishing-cross-terms}
    \< H_{N_k}   a_k, b_k \> \to 0 \qquad \text{and} \qquad
    (D_\rho a_k, D_\sigma b_k)_{\ell^2(\L_{N_k})} \to 0
    \qquad \text{as }  k \to \infty,
  \end{equation}
  for all $\rho, \sigma \in \Rc$.
  We split
  \[
      \< H_{N_k} v_k, v_k \>
      =
      \< H_{N_k} a_k, a_k\> + 2 \< H_{N_k} a_k, b_k \> + \< H_{N_k} b_k, b_k \>.
  \]
  According to \eqref{eq:proof-infsup:vanishing-cross-terms},
  the cross-term vanishes in the limit, $ \< H_{N_k} a_k, b_k \> \to 0$.

   Since the support of $Db_k$ does not intersect the defective region it is
   easy to see that
  \[
    \< H_{N_k} b_k, b_k \>
    =
    \big\<\delta^2 \E_{N_k}^{\rm hom}(u_{N_k}^{\rm hom}) S_{N_k}^{\rm hom} b_k, S_{N_k}^{\rm hom} b_k \big\>.
  \]
  where $u_{N_k}^{\rm hom} := S_{N_k}^{\rm hom} u_{N_k}$. Next, we observe that
  \begin{align*}
    &\Big|\big\< \big[\delta^2 \E_{N_k}^{\rm hom}(u_{N_k}^{\rm hom})
          - \delta^2 \E_{N_k}^{\rm hom}(0) \big]
           S_{N_k}^{\rm hom} b_k, S_{N_k}^{\rm hom} b_k \big\>\Big| \\
    =~& \bigg|\sum_{\ell \in \L_{N_k}^{\rm hom} \setminus \L_{R_k/2}^{\rm hom}}
        \Big( \nabla^2 V(D u_{N_k}^{\rm hom}(\ell)) - \nabla^2 V(0) \Big)
        \big[ DS_{N_k}^{\rm hom} b_k(\ell), DS_{N_k}^{\rm hom} b_k(\ell)\big] \bigg| \\
    \lesssim~&
    \epsilon_k \| D S_{N_k}^{\rm hom} b_k \|_{\ell^2(\L^{\rm per, hom}_{N_k})}^2
    \lesssim
    \epsilon_k \| b_k \|_{\WHper_{N_k}}^2,
  \end{align*}
  where, according to \eqref{eq:TR_estimates:err1},
  \begin{align*}
  \epsilon_k &:= \max_{\ell \in \L_{N_k}^{\rm hom} \setminus \L_{R_k/2}^{\rm hom}} \big| \Dh u_{N_k}^{\rm hom}(\ell) \big|\\
  &=\max_{\ell \in \L_{N_k}\setminus \L_{R_k/2}} \big| D T_{N_k} u(\ell) \big|\\
  &\leq \lVert D T_{N_k} u \rVert_{\ell^2(\L_{N_k}\setminus \L_{R_k/2})} \\
  &\leq \lVert D u \rVert_{\ell^2(\L_{N_k}\setminus \L_{R_k/2})} + \lVert D T_{N_k} u -Du \rVert_{\ell^2(\L_{N_k})} \\
  &\lesssim \lVert D u \rVert_{\ell^2(\L_{N_k}\setminus \L_{R_k/2})}\\
  &\to 0 \quad \text{ as } k \to \infty.
  \end{align*}
  In particular, using Lemma \ref{lem:periodicstable}, we obtain
  \begin{align}
    \notag
    \< H_{N_k} b_k, b_k \>
    &\geq
    \big\< \delta^2 \E_{N_k}^{\rm hom}(0) S_{N_k}^{\rm hom} b_k, S_{N_k}^{\rm hom} b_k \big\>
    - C\epsilon_k \| b_k\|_{\WHper_{N_k}}^2 \\
    &\geq
    (c_0 - C\epsilon_k) \| b_k\|_{\WHper_{N_k}}^2
    \label{eq:infsup-prf:HNbNbN}
    \geq
    c_0/2 \| b_k\|_{\WHper_{N_k}}^2,
  \end{align}
  for $k$ sufficienty large.

  Finally, since $Da_k$ are supported in $\L_{{N_k}/4}$ we have
 \begin{align*}
 (T_{N_k} a_k, \psi_j)_{\WH}
      &=
       (a_k, T_{N_k,N_k/2}^{\rm per}\psi_j)_{\WHper_{N_k}} + (T_{N_k} a_k, (I-T_{N_k/2})\psi_j)_{\WH}\\
     &= (v_k, T_{N_k,N_k/2}^{\rm per}\psi_j)_{\WHper_{N_k}}
      -(b_k, T_{N_k,N_k/2}^{\rm per}\psi_j)_{\WHper_{N_k}}  + 0 \\
    &= -(b_k, T_{N_k,N_k/2}^{\rm per}\psi_j)_{\WHper_{N_k}}
 \end{align*}
 Hence, for all $j = 1, \dots, q$,
\begin{align*}
\lvert (T_{N_k} a_k, \psi_j)_{\WH} \rvert &= \lvert (b_k, T_{N_k,N_k/2}^{\rm per}\psi_j)_{\WHper_{N_k}} \rvert \\
&\lesssim \max_j \lVert D\psi_j \rVert_{\ell^2(\Lambda_{N_k} \setminus \Lambda_{R_k/2})}
 =: \gamma_k \to 0.
\end{align*}
  Writing $T_{N_k} a_k = a_k^+ + a_k^-$ with $a_k^\pm \in \WHH_\pm $, Lemma~\ref{th:infsup-proof:prelim}(ii) implies
  \begin{equation}     \label{eq:infsup-prf:HNaNaN}
      \< H_{N_k} a_k, a_k \>
      =
      \< H T_{N_k} a_k, T_{N_k} a_k \> \geq c_1 \| a_k^+ \|_{\WH}^2 - C \| a_k^- \|_{\WH}^2  \geq c_1 \| a_k \|_{\WHper_{N_k}}^2 - C \gamma_k.
  \end{equation}

  In summary, combining \eqref{eq:infsup-prf:HNaNaN} with
  \eqref{eq:infsup-prf:HNbNbN} and  recalling again
  \eqref{eq:proof-infsup:vanishing-cross-terms} we conclude that, for
  $k$ sufficiently large,
  \begin{align*}
  \inf_{\substack{v \in \WHH_{N_k,+} \\ \lVert v \rVert_{\WHper_{N_k}} = 1}}
        \< H_{N_k} v, v \>
      & = \< H_{N_k} v_k, v_k \> \\[-7mm]
      &\geq
       \min(c_0/2, c_1)
          \big( \| a_k \|_{\WHper_{N_k}}^2 + \| b_k\|_{\WHper_{N_k}}^2 \big) -C \gamma_k  \\
      &= \min(c_0/2, c_1) \| v_k \|_{\WHper_{N_k}}^2
          - 2\min(c_0/2, c_1)\, ( a_k, b_k )_{\WHper_{N_k}}  -C \gamma_k \\
      &\geq \min(c_0/4, c_1/2).
  \end{align*}
  Due to the choice of the $N_k$, this means that for all $N$ large enough,
  \begin{align*}
  \inf_{\substack{v \in \WHH_{N,+} \\ \lVert v \rVert_{\WHper_{N}} = 1}}
        \< H_{N} v, v \>
      &\geq \min(c_0/8, c_1/4). \qedhere
  \end{align*}
  %
  %
\end{proof}

As an immediate corollary of the inf-sup stability of the supercell approximation
we obtain a convergence result in the energy norm.

\begin{theorem} \label{th:convergence_Enorm}
  Let $\us \in \Wi$ be an inf-sup stable solution to \eqref{eq:equil}, then
    there exist $C > 0$ such that, for $N$ sufficiently large, there are $\bar{u}_N \in \Wper_N$ satisfying \eqref{eq:results:supercellapprox} as well as
  \[
      \big\| D\us_N - D T_N^{\rm per}\us \big\|_{\ell^2} \lesssim N^{-d/2}.
  \]
\end{theorem}
\begin{proof}
  The proof of this result is identical to that of \cite[Thm. 2.6]{EOS2016},
  replacing positivity of $\delta^2\E_N(T_N^{\rm per} \us)$ with the new inf-sup stability
  result provided by Lemma~\ref{th:inf-sup-N}.
\end{proof}

\subsection{Uniform Convergence}
\label{sec:uniform}
Let $\us \in \Wi$ be an inf-sup stable solution to \eqref{eq:equil}, then
according to Theorem~\ref{th:convergence_Enorm} there exists $\us_N \in \Wper_N$ solving
\eqref{eq:results:supercellapprox} and satisfying
\[\lVert D \bar{u} - D \bar{u}_N\rVert_{\ell^2} \lesssim N^{-d/2}.\]

Throughout the remainder of this section, we
 fix $\us$ and the sequences $\us_N$,
as well as $v_N := T_{N}^{\rm per} \us, \us_N^{\rm hom} := S_N^{\rm hom} \us_N,
v_N^{\rm hom} := S_N^{\rm hom} v_N$,
\[
  e_N:= \bar{u}_N - v_N, \qquad \text{and} \qquad
  e_N^{\rm hom} := \us_N^{\rm hom} - v_N^{\rm hom}.
\]
In particular we will also assume implicitly that $N$ is sufficiently large so
that the existence of $\us_N$ is guarenteed. We will prove a uniform convergence
rate for $e_N$, from which Theorem~\ref{th:main theorem} will readily follow.

Recall from the definition of $T_N^{\rm per}$ and from
Corollary~\ref{th:decay_TN_ubar} that
\begin{equation} \label{eq:convprf:properties_vN}
\begin{split}
    Dv_N = D\bar{u} \quad &\text{ in } \L_{N/2}, \qquad
    \text{and}  \\
    |D^j v_N| \lesssim N^{1-d-j} \quad &\text{ in } \L_N \setminus \L_{N/3},
    \quad j = 1, 2.
\end{split}
\end{equation}

First, we use our Green's function estimates to
obtain an implicit estimate for $D e_N$.

\begin{lemma} \label{lem:pointwise}
There exists $r_0,C_1>0$ such that for all $\ell \in \L_N \setminus \L_{r_0}$
  \[
    |D e_N(\ell)| \leq C_1 \bigg( N^{-d} + \sum_{m \in \L_N} \big({\rm dist}(\ell-m, 2N \mB \Z^d)+1\big)^{-d}(1+ |m|)^{-d} |D e_N(m)| \bigg).
  \]
\end{lemma}
\begin{proof}
Let
\begin{align*}
 \sigma_N^{\rm hom}(\ell) &:= \big(\nabla V(\Dh \bar{u}_N^{\rm hom}(\ell))
      - \nabla V(\Dh v_N^{\rm hom}(\ell))\big) \chi_{B_{\Rdef + \rcut}}(\ell), \\
 \sigma_N^{\rm def}(\ell) &:= \big(\nabla V_\ell(D \bar{u}_N(\ell))
      - \nabla V_\ell(D v_N(\ell)) \big) \chi_{\Rdef + \rcut}(\ell),  \\
 f_N^{\rm bdry}(\ell) &:= -{\rm Div} \nabla V_\ell(D v_N(\ell)) \\
 &:= \sum_{\rho \in -\Rc_\ell} \nabla_{D_\rho} V_{\ell - \rho}(Dv_N(\ell - \rho)) - \sum_{\rho \in \Rc_\ell} \nabla_{D_\rho} V_{\ell}(Dv_N(\ell)),
\end{align*}
where $\nabla_{D_\rho} V_\ell(Du(\ell)) = \partial V_\ell(Du(\ell)) / \partial
D_\rho u(\ell)$.

Then a straightforward algebraic manipulation shows that,
 for any $w \in \Wper_N(\Lhom_N)$,
\begin{align*}
\big\langle \delta \E_N^{\rm hom}(\bar{u}_N^{\rm hom}) &- \delta \E_N^{\rm hom}(v_N^{\rm hom}),w \big \rangle\\
&= (\sigma_N^{\rm hom}, Dw )_{\ell^2(\L_N^{\rm hom})} - (\sigma_N^{\rm def}, DS_N^{\rm def} w )_{\ell^2(\L_N)} - (f_N^{\rm bdry}, S_N^{\rm def} w)_{\ell^2(\L_N)}.
\end{align*}
Furthermore, for $N \geq 6 r_{\rm cut}$, it is straightforward
to establish that
\begin{align}
  \label{eq:prfconv:step1:sigNhom}
 \lvert \sigma_N^{\rm hom}(\ell) \rvert &\lesssim  \lvert \Dh e_N^{\rm hom}(\ell) \rvert
    && \hspace{-2cm} \forall \ell \in \Lhom_N \cap B_{\Rdef + \rcut},  \\
    \label{eq:prfconv:step1:sigNdef}
 \lvert  \sigma_N^{\rm def}(\ell) \rvert &\lesssim \lvert D e_N(\ell) \rvert,
      && \hspace{-2cm} \forall \ell \in \L_N \cap B_{\Rdef + \rcut}, \qquad \text{and} \\
      \label{eq:prfconv:step1:fbdry}
 \lvert f_N^{\rm bdry}(\ell)\rvert & \lesssim
    \begin{cases}
        0, & \ell \in \L_{N/3}, \\
        N^{-d-1}, & \ell \in \L_N \setminus \L_{N/3}.
    \end{cases}
\end{align}
The first two estimates follow simply from the fact that $V, V_\ell \in C^4$
while \eqref{eq:prfconv:step1:fbdry} follows from the second-order difference
structure of $f_N^{\rm bdry}(\ell)$ and \eqref{eq:convprf:properties_vN}.

Furthermore, Taylor expansions of $\delta \E_N^{\rm hom}(\us_N^{\rm hom})$ and
$\delta \E_N^{\rm hom}(v_N^{\rm hom})$ about $0$, and some elementary
manipulations yield
\begin{align*}
&\big\langle \delta \E_N^{\rm hom}(\bar{u}_N^{\rm hom}) - \delta \E_N^{\rm hom}(v_N^{\rm hom}),w \big\rangle \\
&= \big\langle \delta^2 \E_N^{\rm hom}(0) e_N^{\rm hom}, w \rangle
  + \int_0^1 \big\< \big[ \delta^2\E_N^{\rm hom}(t v_N^{\rm hom})
                        - \delta^2 \E_N^{\rm hom}(0) \big]
                    e_N^{\rm hom}, w \big\> \, dt  \\
& \hspace{3.7cm}
  + \int_0^1 \big\< \big[ \delta^2\E_N^{\rm hom}(t \us_N^{\rm hom})
                        - \delta^2 \E_N^{\rm hom}(t v_N^{\rm hom}) \big]
                    u_N^{\rm hom}, w \big\> \, dt \\
&= \big\langle \delta^2 \E_N^{\rm hom}(0) e_N^{\rm hom}, w \rangle \\
& \qquad + \int_0^1 \int_0^t \big\langle \big[
      \delta^3 \E_N^{\rm hom}( s v_N^{\rm hom} )e_N^{\rm hom}, v_N^{\rm hom}, w \big\rangle\,dt\,ds\\
& \qquad + \int_0^1 \int_0^t \big\langle \delta^3 \E_N^{\rm hom}( (t-s) v_N^{\rm hom} + s \us_N^{\rm hom})
      e_N^{\rm hom}, e_N^{\rm hom} + v_N^{\rm hom}, w \big\rangle\,dt\,ds.
\end{align*}
We test with $w(n) = \Dh G_N(n -\ell)$ then
\begin{equation} \label{eq:prfconv:step1:estimate_Djw}
    |(\Dh)^j w(n)| \lesssim \big({\rm dist}(\ell-m, 2N \mB \Z^d)+1\big)^{1-d-j},
\end{equation}
hence, for  $\lvert \ell \rvert > \Rdef + \rcut$, we obtain
\begin{align*}
|De_N(\ell)| &= \lvert \Dh e_N^{\rm hom}(\ell) \rvert
= \big\langle \delta^2 \E_N^{\rm hom}(0) (\bar{u}_N^{\rm hom} - v_N^{\rm hom}) ,w \rangle  \\
&\lesssim
  \big| \big( \sigma_N^{\rm hom}, \Dh w \big)_{\ell^2(\Lhom_N)} \big|
   + \big| \big( \sigma_N^{\rm def}, DS_N^{\rm def} w \big)_{\ell^2(\L_N)} \big|
   + \big\lvert (f_N^{\rm bdry}, S_N^{\rm def} w)_{\ell^2(\L_N)} \big\rvert  \\
 &\quad
   + \sum_{m \in \Lhom_N} |De_N^{\rm hom}(m)|^2 \, |Dw(n)| \\
&\quad
  + \sum_{m \in \Lambda_N^{\rm hom}} \lvert D e_N^{\rm hom}(m) \rvert
  \lvert Dw(m) \rvert \, |Dv_N^{\rm hom}(m)| \\
&=: {\rm T}_{1} + {\rm T}_{2} + {\rm T}_{3} + {\rm T}_4 + {\rm T}_5.
\end{align*}
The fifth term is already of the form we require: We can employ
\eqref{eq:prfconv:step1:estimate_Djw} to bound $Dw$ and
\eqref{eq:convprf:properties_vN} to bound $Dv_N^{\rm hom}$. Furthermore, we use
Lemma~\ref{th:properties_S} to bound $De_N^{\rm hom}$ by $De_N$ to arrive at
\begin{align*}
  {\rm T}_5
&\lesssim  \sum_{m \in \Lambda_N^{\rm hom}} \lvert D e_N(m) \rvert \big({\rm dist}(\ell-m, 2N \mB \Z^d)+1\big)^{-d} (1+|m|)^{-d}
\end{align*}
for $\lvert \ell \rvert > 2 \RS$. Using, \eqref{eq:prfconv:step1:fbdry} in combination with \eqref{eq:prfconv:step1:estimate_Djw}, as well as using $\lVert De_N^{\rm hom} \rVert_{\ell^2} \leq N^{-d/2}$ we get
$
  {\rm T}_3 +{\rm T}_4 \lesssim N^{-d}.
$
Finally, for ${\rm T}_1$ and ${\rm T}_2$ and again, $\lvert \ell \rvert > 2 \RS$, we use \eqref{eq:prfconv:step1:sigNhom} and \eqref{eq:prfconv:step1:sigNdef} to estimate
\begin{align*}
{\rm T}_1 + {\rm T}_2 &\lesssim \hspace{-3mm} \sum_{m \in \Lhom_N \cap B_{\RS}} \hspace{-3mm} \lvert \Dh e_N^{\rm hom} (m) \rvert \lvert D^2 \G (\ell -m) \rvert + \hspace{-3mm} \sum_{m \in \L_N \cap B_{\RS}} \hspace{-3mm} \lvert D e_N (m) \rvert \lvert D S^{\rm def} \Dh \G (\ell -m) \rvert \\
&\lesssim \hspace{-3mm} \sum_{k \in \Lhom_N \cap B_{\RS}} \sum_{m \in \L_N \cap B_{\RS}} \lvert D e_N (m) \rvert \lvert D^2 \G (\ell -k) \rvert \\
&\lesssim \sum_{m \in \Lambda_N^{\rm hom}\cap B_{\RS}} \lvert D e_N(m) \rvert \big({\rm dist}(\ell-m, 2N \mB \Z^d)+1\big)^{-d}. \qedhere
\end{align*}

%
\end{proof}

Next, we prove a discrete Caccioppoli estimate.

\begin{lemma} \label{lem:ell2}
  There exist $r_1, C_2 > 0$ such that, for $r_1 \leq r \leq N/4$,
  \[
    \| De_N \|_{\ell^2(\L_{r/2})} \leq C_2 \| De_N\|_{\ell^2(\L_{2r} \setminus \L_{r/2})}.
  \]
\end{lemma}
\begin{proof}
  Inf-sup stability of $v_N$ established in Lemma~\ref{th:inf-sup-N} and the
  convergence $\|D e_N \|_{\ell^2} \to 0$ (cf. Theorem~\ref{th:convergence_Enorm}) imply that there exists $c_1 > 0$ such that,
  for all $N$ sufficiently large,
  \[
      \sup_{\substack{w \in \Wper_N \\ \|Dw\|_{\ell^2} = 1}}
       \int_0^1 \big\langle \delta^2\E_N(v_N + t e_N) z, w \big\rangle \, dt
       \geq c_1 \lVert Dz \rVert_{\ell^2} \qquad \forall z \in \Wper_N.
  \]
  In the rest of this proof we will write $\sup_w = \sup_{w \in \Wper_N,
  \|Dw\|_{\ell^2} = 1}$ . Fix $r > 0$ and insert $z = T_{N,r}^{\rm per} e_N$ in
  the inf-sup condition, then we can use the fact that the supports of $Dz$ and
  $D(I-T_{N,2r}^{\rm per}) w$ do not overlap to write
  \begin{align*}
    \| D z \|_{\ell^2}
    &\lesssim \sup_{w}
    \int_0^1 \big\< \delta^2\E_N(v_N + t e_N) z, w \big\> \,dt \\
    &= \sup_{w}
    \int_0^1 \big\< \delta^2\E_N(v_N + t e_N) z, T_{N,2r}^{\rm per} w \big\> \,dt \\
    &= \sup_{w} \bigg( \int_0^1 \big\< \delta^2\E_N(v_N + t e_N) (T_{N,r}^{\rm per}-I) e_N, T_{N,2r}^{\rm per} w \big\> \,dt \\
    & \hspace{3cm} + \big\< \delta\E_N(\us_N) - \delta\E_N(v_N), T_{N,2r}^{\rm per} w \big\>\bigg).
  \end{align*}
  We clearly have $\< \delta\E_N(\us_N), T_{N,2r}^{\rm per} w \> = 0$. Moreover,
  \begin{align*}
    \big\< \delta\E_N(v_N), T_{N,2r}^{\rm per} w \big\>
    &= \sum_{\ell \in \L_{2r}} \nabla V_\ell(Dv_N(\ell))\big[DT_{N,2r}^{\rm per} w\big] \\
    &= \sum_{\ell \in \L_{2r}} \nabla V_\ell(D\us(\ell))\big[DT_{2r} w\big] \\
    &= \big\< \delta \E(\us), T_{2r} w \big\> = 0,
  \end{align*}
  which leaves us with only the term
  \begin{align*}  
      \| D z \|_{\ell^2}
      & \lesssim
      \sup_{w} \int_0^1 \big\< \delta^2\E_N(v_N + t e_N) (T_{N,r}^{\rm per}-I) e_N, T_{N,2r}^{\rm per} w \big\> \,dt.
  \end{align*}
  Since $DT_{N,2r}^{\rm per} w  = 0$ in $\L \setminus \L_{2r}$ we can estimate this further
  by
  \begin{align*}
    \| D z \|_{\ell^2}
    &\lesssim
    \sup_w  \sum_{\ell \in \L_{2r}}
            \big| D(T_{N,r}^{\rm per}-I) e_N(\ell) \big| \,
            \big| D T_{N,2r}^{\rm per} w(\ell) \big| \\
    &\lesssim \sup_w  \|  D(T_{N,r}^{\rm per}-I) e_N(\ell) \|_{\ell^2(\L_{2r})} \,
        \| D T_{N,2r}^{\rm per} w \|_{\ell^2} \\
    &\lesssim
    \|  D(T_{N,r}^{\rm per}-I) e_N(\ell) \|_{\ell^2(\L_{2r})},
  \end{align*}
  where, in the last estimate, we used Lemma~\ref{th:TR_estimates}
   to bound $\| D
  T_{N,2r}^{\rm per} w \|_{\ell^2} \lesssim \|Dw\|_{\ell^2} \lesssim 1$.

  Using Lemma~\ref{th:TR_estimates} a second time we finally deduce that
  \begin{align*}
  	\| D e_N \|_{\ell^2(\L_{r/2})} &\leq \| D z \|_{\ell^2(\L_N)}
  	\lesssim \|  D(T_{r}-I) e_N(\ell) \|_{\ell^2(\L_{2r})} \\
  	&\lesssim \|  D(T_{r}-I) e_N(\ell) \|_{\ell^2(\L_{r})} + \|  De_N(\ell) \|_{\ell^2(\L_{2r}\setminus \L_{r})}\\
  	&\lesssim \|  De_N(\ell) \|_{\ell^2(\L_{2r}\setminus \L_{r/2})}. \qedhere
  \end{align*}
\end{proof}

Our main result, Theorem~\ref{th:main theorem}, will follow
from the next intermediate result, which is of independent interest.

\begin{theorem} \label{thm:ellinftyregularity}
  Under the conditions of Theorem~\ref{th:main theorem},
  \[
      \lVert D e_N \rVert_{\ell^\infty(\L_N)} \lesssim N^{-d}.
  \]
\end{theorem}
\begin{proof}
Let $\omega(r) := \| De_N \|_{\ell^\infty(\L_N \setminus
\L_r)}$, then according to Lemma \ref{lem:pointwise}, for
$r \geq r_0$,
\begin{align*}
\omega(r) &\lesssim N^{-d} + \sup_{\ell \in \L_N \setminus \L_r}
    \sum_{m \in \Lambda_N} \big({\rm dist}(\ell-m, 2N \mB \Z^d)+1\big)^{-d}
        (1+\lvert m \rvert)^{-d} \lvert D e_N(m) \rvert \\
&\lesssim N^{-d} + \omega(r)
    \sup_{\ell \in \L_N \setminus \L_r}
    \sum_{m \in \Lambda_N\setminus\L_r} \big({\rm dist}(\ell-m, 2N \mB \Z^d)+1\big)^{-d} (1+ |m|)^{-d} \\
& \hspace{2.1cm}
 + \sup_{\ell \in \L_N \setminus \L_r}
    \sum_{m \in \L_r} (1+\lvert \ell - m \rvert)^{-d}
    (1+\lvert m \rvert)^{-d} \lvert D e_N(m) \rvert\\
    &\lesssim N^{-d} + \omega(r)
    \sup_{\ell \in \L_N \setminus \L_r} \sum_{z \in \{-1,0,1\}^d }
    \sum_{m \in\Lambda_N} (1+\lvert \ell-m -2N \mB z \rvert)^{-d} (1+ |m|)^{-d} \\
& \hspace{2.1cm}
 + \sup_{\ell \in \L_N \setminus \L_r}
    \Big( \sum_{m \in \L_r} (1+\lvert \ell - m \rvert)^{-2d}
    (1+\lvert m \rvert)^{-2d}\Big)^{1/2} \lVert D e_N \rVert_{\ell^2(\L_r)}\\
&\lesssim N^{-d} + \omega(r)
    \sup_{\ell \in \L_N \setminus \L_r} \max_{z \in \{-1,0,1\}^d } |\ell-2N \mB z|^{-d} \log|\ell-2N \mB z|\\
& \hspace{2.1cm}
 + \sup_{\ell \in \L_N \setminus \L_r}
    |\ell|^{-d} \lVert D e_N \rVert_{\ell^2(\L_r)}\\
&\lesssim  N^{-d} + \omega(r) r^{-d} \log(r) + r^{-d} \lVert D e_N \rVert_{\ell^2(\L_r)} .
\end{align*}
Here we used that, for $|\ell| \geq 2$,
\begin{align*}
\sum_{m \in \L} (1+\lvert \ell-m \rvert)^{-d} (1+ |m|)^{-d} &\lesssim |\ell|^{-d} \log|\ell|, \qquad \text{ and}\\
\sum_{m \in \L} (1+\lvert \ell-m \rvert)^{-2d} (1+ |m|)^{-2d} &\lesssim |\ell|^{-2d}.
\end{align*}

We apply the Caccioppoli inequality, Lemma~\ref{lem:ell2}, further
restricting to $r_1/2 \leq r \leq N/8$, to continue to estimate
\begin{align*}
\omega(r)
&\lesssim N^{-d} + \omega(r) r^{-d} \log(r) + \lVert D e_N \rVert_{\ell^2(\L_{4r} \setminus \L_r)} r^{-d}\\
&\leq C_3\big( N^{-d} + \omega(r) r^{-d/2} \big).
\end{align*}
For $r_2 := (2 C_3)^{2/d}$ and $r \geq r_3 := \max\{r_0, r_1, r_2\}$, we thus
find $\omega(r) \leq 2C_3 N^{-d}$. That is, we have proven that $|De_N(\ell)|
\lesssim N^{-d}$ for all $\ell \in \L_N \setminus \L_{r_3}$, where $r_3$ is
independent of $N$.

It thus remains only to consider $\ell \in \L_{r_3}$, a finite subdomain.
Applying Lemma~\ref{lem:ell2} a second time we obtain
\[
  \lvert De_N(\ell) \rvert \leq \lVert De_N \rVert_{\ell^2(\L_{r_3})}
   \lesssim  \lVert De_N \rVert_{\ell^2(\L_{4 r_3} \setminus \L_{r_3})}
   \lesssim \omega(r_3)
   \lesssim N^{-d}. \qedhere\]
\end{proof}

\begin{proof}[Proof of Theorem~\ref{th:main theorem}]
  We split
  \begin{align*}
      \| D\us_N - D\us \|_{\ell^\infty(\L_N)}
      &\leq \| D e_N \|_{\ell^\infty(\L_N)}
        + \| D v_N - D \us \|_{\ell^\infty(\L_N)} \\
      &\lesssim N^{-d} + \| D v_N - D \us \|_{\ell^\infty(\L_N)},
  \end{align*}
  where we used Theorem~\ref{thm:ellinftyregularity}.
  For $\ell \in \L_{N/2}$, $D v_N(\ell) - D \us(\ell) = 0$.
  Conversely, for $\ell \in \L_N \setminus \L_{N/2}$,
  \eqref{eq:convprf:properties_vN} and \eqref{eq:decay_ubar} imply that
  \[
      |Dv_N(\ell) - D\us(\ell) | \leq |D v_N(\ell)| + |D \us(\ell)| \lesssim N^{-d}. \qedhere
  \]
\end{proof}

\bibliographystyle{alpha}
\bibliography{bib}

\end{document}